\documentclass[reqno,12pt]{amsart}
\usepackage{amsmath,amsthm,amsfonts,amssymb}
\usepackage{bm}
\usepackage{ytableau}
\usepackage{euscript, mathrsfs} 
\usepackage{tikz}
\usepackage[colorlinks]{hyperref}
\usepackage[margin=1in]{geometry}
\numberwithin{equation}{section}    
\usepackage[all]{xy}
\usepackage{graphicx,import}
\usepackage{amscd}
\usepackage{color}
\usepackage{enumerate}
\usepackage{fourier}
\usepackage{tikz}   
\usepackage{cases}

\theoremstyle{plain}
\newtheorem{thm}{Theorem}[section]

\newtheorem{proposition}[thm]{Proposition}

\newtheorem{corollary}[thm]{Corollary}
\newtheorem{question}[thm]{Question}

\def\diag{ \begin{tikzpicture} \draw[dashed] (-.12,-.12) -- (.42, .42); \end{tikzpicture} }

\newcommand{\bmlambda}{{\bm{\lambda}}}
\newcommand{\bmmu}{{\bm{\mu}}}
\newcommand{\bmnu}{{\bm{\nu}}}
\newcommand{\bmT}{{\bm{T}}}
\newcommand{\bmR}{{\bm{R}}}

\newcommand{\QQ}{{\mathbb {Q}}}

\DeclareMathOperator{\maj}{maj}
\DeclareMathOperator{\arm}{arm}
\DeclareMathOperator{\leg}{leg}
\DeclareMathOperator{\coleg}{coleg}
\DeclareMathOperator{\coarm}{coarm}
\DeclareMathOperator{\inv}{inv}

\DeclareMathOperator{\ssyt}{SSYT}
\DeclareMathOperator{\syt}{SYT}
\DeclareMathOperator{\llt}{LLT}

\title[Haglund's conjecture for
multi-$t$ Macdonald polynomials]{Haglund's conjecture for multi-$t$ Macdonald polynomials}

\author{Seung Jin Lee}
\address{Department of Mathematical Sciences \\ Research institute of Mathematics \\ Seoul National University \\ Seoul 151-747 \\ Korea}
\email{lsjin@snu.ac.kr}

\author{Jaeseong Oh}
\address{School of Computational Sciences \\ Korea Institute for Advanced Study \\ Seoul 02445 \\ Korea}
\email{jsoh@kias.re.kr}

\author{Brendon Rhoades}
\address{Department of Mathematics\\ University of California, San Diego\\United States}
\email{bprhoades@math.ucsd.edu}

\begin{document}
\begin{abstract}
    We provide new approaches to prove identities for the modified Macdonald polynomials
    via their LLT expansions. As an application, we prove a conjecture of Haglund concerning the multi-$t$-Macdonald polynomials of two rows.
    \end{abstract}
\maketitle

\section{Introduction}

In his seminal paper \cite{Mac88}, Macdonald introduced the \emph{Macdonald $P$-polynomials} $P_\mu[X;q,t]$ indexed by partitions $\mu$. The \emph{modified Macdonald polynomials} $\widetilde{H}_\mu[X;q,t]$ are a combinatorial version of the Macdonald $P$-polynomials and they are characterized as the 
unique family of symmetric functions 
satisfying the following \emph{triangularity}
and \emph{normalization} axioms (see \cite{HHL05}):
\begin{enumerate}
    \item $\widetilde{H}_\mu[X(1-q);q,t]=\sum_{\lambda\ge\mu} a_{\lambda,\mu}(q,t)s_\lambda(X)$,
    \item $\widetilde{H}_\mu[X(1-t);q,t]=\sum_{\lambda\ge\mu'} b_{\lambda,\mu}(q,t)s_\lambda(X)$, and
    \item $\langle \widetilde{H}_\mu, s_{(n)}\rangle=1$,
\end{enumerate}
for suitable coefficients $a_{\lambda,\mu}, b_{\lambda,\mu}\in\QQ(q,t)$, where $\mu'$ denotes the conjugate partition of $\mu$ and $s_\mu(X)$ is the Schur function. The partial order $\le$ is the
\emph{dominance order} on partitions defined by
\begin{equation*}
    \lambda\le\mu \text{ if } \lambda_1+\cdots+\lambda_k\le\mu_1+\cdots+\mu_k \text{ for all } k,
\end{equation*} $[-]$ denotes the plethystic substitution, and $\langle -,-\rangle$ is the Hall inner product.
Haglund, Haiman,
and Loehr proved \cite{HHL05} 
a combinatorial formula for the modified Macdonald polynomials $\widetilde{H}_\mu[X;q,t]$ which generalizes
to the \emph{multi-$t$ Macdonald polynomials}
$\widetilde{H}_{\mu}[X;q,t_1,t_2,\dots]$.
The polynomial $\widetilde{H}_{\mu}[X;q,t_1,t_2,\dots]$
specializes to $\widetilde{H}_\mu[X;q,t]$ at 
$t_1 = t_2 = \cdots = t$ and
depends on an order $c_1, c_2, \dots$ of the cells of $\mu$.

{\em LLT polynomials} are symmetric functions
$\llt_{\bmnu}[X;q]$ introduced by Lascoux,
Leclerc, and Thibon \cite{LLT97},
which depend  
on a tuple $\bmnu$ of skew partitions. 
The LLT polynomial $\llt_{\bmnu}[X;q]$ is 
{\em unicellular} if every skew partition in $\bmnu$ is a single cell.
Unicellular LLT polynomials are naturally indexed by Dyck paths as well as tuples of skew shapes.

Jim Haglund conjectured \cite{Hag21} 
a combinatorial formula expanding the multi-$t$ Macdonald polynomials indexed by 
two-row partitions $\mu$ into unicellular LLT polynomials.
In this paper, we prove Haglund's conjecture.
In the following theorem, we index LLT polynomials with Dyck paths.

\begin{thm}\label{thm: Haglund's conjecture}
Let $\mu=(n-k,k)$ be a two-row partition and $c_1, \dots,c_k$ be the cells in the upper row. Let $D(h_1,\dots,h_n)$ be the Dyck path of size $n$ whose height of the $j$-th column is $h_j$ for $1\le j \le n$. Then for $k\le h_1\le \cdots\le h_k$ we have, 
\[
\widetilde{H}_\mu[X;q,q^{h_k-k}, q^{h_{k-1}-k}, \dots, q^{h_1-k}] = \llt_{D(h_1,\dots,h_k,n,\dots,n)}[X;q],
\]
where the left-hand side is the multi-$t$-Macdonald polynomial $\widetilde{H}_\mu[X;q,t_{1},t_{2},\dots]$ at $t_{i}=q^{h_i-k}$ for $1\le i \le k$.
\end{thm}

This paper is organized as follows. In Section~\ref{Sec: background}, we provide background on the combinatorics of LLT and modified Macdonald polynomials. 
Section~\ref{Sec: A combinatorial proof} contains various equivalences of LLT polynomials indexed
by different families of skew shapes and a proof of Theorem~\ref{thm: Haglund's conjecture} based on these
equivalences.
In Section~\ref{Sec: Stretching}, we explore 
`stretching symmetries' of modified Macdonald polynomials 
which bear formal similarity to
Theorem~\ref{thm: Haglund's conjecture}.
We close in Section~\ref{Sec: Conclusion} with some 
open problems.

\section{background}\label{Sec: background}

\subsection{Combinatorics}
Let $\mu=(\mu_1,\mu_2,\dots,\mu_\ell)$ be a partition of $n$. We identify $\mu$ with its (French)
{\em Young diagram}
\[\mu=\{(i,j)\in \mathbb{Z}_+\times \mathbb{Z}_+ : j\le \mu_i\}\] 
and refer to elements in $\mu$ as {\em cells}.
For a cell $u$ in a partition, 
\begin{itemize}
    \item the {\em content} of $u = (i,j)$ is 
    $c(u) := i-j$,
    \item the \emph{arm} (resp., \emph{coarm}) of $u$ is the number of cells strictly to the right (resp., left) of $u$ in the same row,
    \item  the \emph{leg} (resp., \emph{coleg}) of $u$
is the number of cells strictly above (resp., below) $u$ in the same column, and
\item the {\em major index} of $u$ is the leg of $u$ plus one. 
\end{itemize}
For example, for a partition $\mu=(5,4,3,2)$, and a cell $c=(2,2)$ as in Figure~\ref{fig: partition}, 
\[
    \arm(c)=2, \coarm(c)=1, \leg(c)=1, \coleg(c)=2, \text{ and } \maj(c)=2.
\]
\begin{figure}[h]
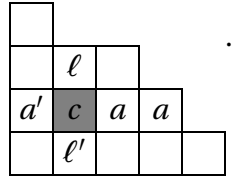

    \centering
    \[
        \begin{ytableau}
          \\
        & \ell &  \\
        a' & *(gray) c & a & a\\
        &  \ell' &  &  &
        \end{ytableau}.
    \]
    \caption{An example of partition}
    \label{fig: partition}
\end{figure}
Let $\textrm{stat}$ be a statistic on cells.
For a subset of cells $D\subseteq \mu$, the number $\textrm{stat}(D)$ is defined by
\begin{align*}
\textrm{stat}(D):=\sum_{u\in D} \textrm{stat}(u).
\end{align*}

For partitions $\lambda$ and $\mu$  with $\mu\subseteq\lambda$, the \emph{skew shape} is the
set-theoretic difference
$\lambda/\mu := \lambda - \mu$. 
 A \emph{ribbon} is an edgewise connected skew shape containing no $2\times 2$ block of cells. Note that the contents of the cells of a ribbon are consecutive integers. The \emph{descent set} of a ribbon $\nu$ is the set of contents $c(u)$ of those cells $u=(i,j)\in\nu$ such that the cell $v=(i-1,j)$ directly below $u$ also belongs to $\nu$. For an interval $I=[r,r+s]:=\{r,r+1,\dots,r+s\}$, 
 there is a one-to-one correspondence between ribbons of content $I$ and subsets $D\subseteq I\setminus\{r\}$ by considering the descent set of each ribbon (we regard
 diagonal translations of ribbons as indistinguishable). 
 We denote the ribbon with a content set $I$ and a descent set $D$ by $R_I(D)$. In particular, 
 for any integer $a$ we use $C_a := R_{\{a\}}(\emptyset)$ to denote a cell of content $a$. 
 For each subset of cells $D\subseteq \{(i,j)\in\mu:1<i\}$ where no cell is in the first row, let $D^{(j)}:=\{i:(i,j)\in D\}$. For a partition $\mu$, and a subset $D$ of $\mu$ without a cell in the first row, $\bmR_\mu(D)$ is a tuple of ribbons defined by
\[
\bmR_\mu(D):=(R_{[1,\mu'_1]}(D^{(1)}),R_{[1,\mu'_2]}(D^{(2)}),\dots).
\]

\subsection{LLT polynomials and modified Macdonald polynomials}\label{Subsec:A combinatorial formula for Macdonald polynomials}
For a skew partition $\nu$, a \emph{semistandard tableau} of shape $\nu$ is a filling of $\nu$ with positive integers where each row is weakly increasing from left to right and each column is strictly increasing from bottom to top. For a tuple $\bmnu=(\nu^{(1)},\nu^{(2)},\dots)$ of skew partitions, a semistandard tableau $\bmT=(T^{(1)},T^{(2)},\dots)$ of shape $\bmnu$ is a tuple of semistandard tableaux where each $T^{(i)}$ is a semistandard tableau of shape $\nu^{(i)}$. The set of semistandard tableaux of shape $\bmnu$ is denoted by $\ssyt(\bmnu)$. For a semistandard tableau $\bmT=(T^{(1)},T^{(2)},\dots)$ of shape $\bmnu$, an \emph{inversion} of $\bmT$ is a pair of cells $u\in\nu^{(i)}$ and $v\in\nu^{(j)}$ such that $T^{(i)}(u)>T^{(j)}(v)$ and either
\begin{itemize}
    \item $i<j$ and $c(u)=c(v)$, or
    \item $i>j$ and $c(u)=c(v)+1$.
\end{itemize}
Denote by $\inv(\bmT)$ the number of inversions in $\bmT$. The \emph{LLT polynomial} $\llt_\bmnu[X;q]$ is defined by
\[
 \llt_\bmnu[X;q]=\sum_{\bmT\in\ssyt(\bmnu)} q^{\inv(\bmT)}X^T.
\]
Here, $X^\bmT:=x_1^{\bmT_1}x_2^{\bmT_2}\cdots$, where $\bmT_i$ is the number of $i$'s in $\bmT$. 

If a tuple $\bmnu$ of skew partitions consists of single cells, the LLT polynomial $\llt_{\bmnu}[X;q]$ is called \emph{unicellular}. Dyck paths can also index unicellular LLT polynomials.
Given a tuple of $n$ cells, index these 
cells in order of increasing content, breaking ties
by indexing from left to right within a given content.
For all $i$, let $h_i$ be the maximal index $j$ such that the $i$-th and $j$-th cell forms an inversion pair, i.e., $i<j$ and either both cells are in the same row or the $j$-th cell is in the next row and strictly to the left of the $i$-th cell. Then the correspondence sends the tuple of cells to a Dyck path, where the height of $i$-th column is given by $h_i$. For example, a tuple of cells in Figure~\ref{AAA} corresponds to the Dyck path of height $(2,4,5,5,5)$.

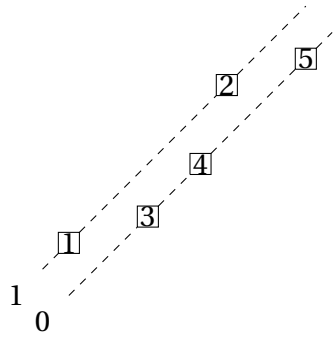
\begin{figure}[h]
\begin{tikzpicture}[scale=0.7]
\draw (0.3-0.5,0.3+0.5)--(0.7-0.5,0.3+0.5)--(0.7-0.5,0.7+0.5)--(0.3-0.5,0.7+0.5)--(0.3-0.5,0.3+0.5);
\draw (1.3,1.3)--(1.7,1.3)--(1.7,1.7)--(1.3,1.7)--(1.3,1.3);
\draw (2.3,2.3)--(2.7,2.3)--(2.7,2.7)--(2.3,2.7)--(2.3,2.3);
\draw (3.3-0.5,3.3+0.5)--(3.7-0.5,3.3+0.5)--(3.7-0.5,3.7+0.5)--(3.3-0.5,3.7+0.5)--(3.3-0.5,3.3+0.5);
\draw (4.3,4.3)--(4.7,4.3)--(4.7,4.7)--(4.3,4.7)--(4.3,4.3);
\draw[dashed] (-0.5,.5)--(-0.2,0.8) (0.2,1.2)--(3.3-0.5,3.3+0.5) (3.7-0.5,3.7+0.5)--(4.5,5.5);
\draw [dashed] (0,0)--(1.3,1.3) (1.7,1.7)--(2.3,2.3) (2.7,2.7)-- (4.3,4.3) (4.7,4.7)--(5,5);
\node[] at (1.5,1.5) {3};
\node[] at (2.5,2.5) {4};
\node[] at (4.5,4.5) {5};
\node[] at (0,1) {1};
\node[] at (3,4) {2};
\node[] at (-1,0) {1};
\node[] at (-1+0.5,-0.5) {0};
\end{tikzpicture}
\caption{Labeling of cells}
\label{AAA}
\end{figure}

Haglund--Haiman--Loehr also provided an expansion of the modified Macdonald polynomials into LLT polynomials indexed by tuples of ribbons.
\begin{thm}\label{thm: LLT expansion of Mac}\cite[Section~3]{HHL05} For a partition $\mu$, we have 
\[
\widetilde{H}_\mu[X;q,t] = \sum_{D} q^{-\arm(D)}t^{\operatorname{maj}(D)}\llt_{\bmR_\mu(D)}[X;q],
\]
where the sum is over all subsets $D\subseteq \{(i, j)\in\mu:1<i\}$.
\end{thm}

As a generalization of the above LLT expansion (or a combinatorial formula) of  modified Macdonald polynomials, the \emph{multi-$t$ Macdonald polynomials} $\widetilde{H}_\mu[X;q,t_1,t_2\dots]$ is defined as follows: For a partition $\mu$, consider an ordering $c_1, c_2, \dots$ of cells in $\mu$. Then the multi-$t$-Macdonald polynomial is 
\begin{equation}\label{Eq: def of multi-t-Mac}
\widetilde{H}_\mu[X;q,t_1,t_2,\dots] := \sum_{D} q^{-\arm(D)}\prod_{c_i \in D} t_i^{\maj(c_i)}\llt_{\bmR_\mu(D)}[X;q],
\end{equation}
where the sum is over all subsets $D\subseteq \{(i, j)\in\mu:1<i\}$. By specializing each $t$ variable $t_i=t$, the multi-$t$-Macdonald polynomial is just a usual modified Macdonald polynomial.

\section{Proving Haglund's conjecture}\label{Sec: A combinatorial proof}

\subsection{LLT-equivalences}\label{Sec: linear relations of LLT's}

Write $[n]_q:= \frac{1-q^n}{1-q}$ for the $q$-analog
of an integer $n$.
Let $\sum_i a_i(q) \bmnu^{(i)}$ and $\sum_j b_j(q) \bmmu^{(j)}$ be  $\mathbb{N}[q]$-linear combinations of tuples of skew partitions. Following Miller \cite{Mil19}, we say these linear combinations are  \emph{LLT-equivalent} if for every tuple $\bmlambda$ of skew partitions, we have
\[
  \sum_i a_i(q) \llt_{(\bmnu^{(i)},\bmlambda)}[X;q]=\sum_j b_j(q) \llt_{(\bmmu^{(j)},\bmlambda)}[X;q].
\]
Here $(\bmnu,\bmlambda)$ is the tuple obtained by concatenating $\bmnu$ and $\bmlambda$. Abusing notation,
we write
\[
\sum_i a_i(q) \bmnu^{(i)}=\sum_j b_j(q) \bmmu^{(j)},
\]
to indicate LLT-equivalence. 
In this section, we establish a series of LLT-equivalences which are ribbon-analogues of results in \cite{Mil19, Lee21, HNY20}.

We prove some of the LLT-equivalences inductively, 
by 
showing that both sides 
satisfy `linear relations'. To be more precise, we say that a function $f(\alpha)$ of integers $\alpha$ satisfies a \emph{linear relation} in $\alpha$ if we have
\[
    q f(\alpha)+f(\alpha+2)=[2]_q f(\alpha+1).
\]
If a function $f(D)$ is defined for Dyck paths $D$, we say $f$ satisfies a \emph{row linear relation} if we have
\[
    q f(D) + f(D'') = [2]_q f(D'),
\]
where Dyck paths $D$, $D'$, and $D''$ are Dyck paths where they differ at only one row, and the number of boxes in that row below the Dyck path is given by $a$, $a+1$ and $a+2$, respectively. We define \emph{column linear relation} similarly.

Let $R$ be a ribbon of content $[r-1]$ with a descent set $D$. We can add a cell of content $r$ to $R$ to obtain a ribbon of content $[r]$ in two ways; add a cell above or to the left of the cell of content $r-1$. In other words, there are two ribbons $R^+_H:=R_{[r]}(D)$ and $R^+_V:=R_{[r]}(D\cup \{r\})$ obtained by adding a cell to $R$. We used the $+$ sign to mean that one cell is added, and $H$ and $V$ stand for adding a cell horizontally or vertically. 
Our first LLT equivalence is as follows.

\begin{proposition}\label{Prop: horizontal + q^a vertical}
For a ribbon $R$ of content $[r-1]$, we have the following LLT equivalence
\begin{align*}
R^+_H + q^\alpha R^+_V&= [\alpha]_q (R, C_r) - q[\alpha-1]_q (C_r, R)\\
&=\begin{cases}
[\alpha]_q (R, C_r) - q[\alpha-1]_q (C_r, R) &\text{ if $\alpha\ge$1, and}\\
-q^\alpha[-\alpha]_q (R, C_r) + q^\alpha[1-\alpha]_q (C_r, R)&\text{ if $\alpha<$1}\\
\end{cases}
\end{align*}
for all integers $\alpha$.
\end{proposition}
\begin{proof}
The formula is easy to see when 
$\alpha=0, 1$. Since each of the three expressions
satisfies a linear relation in $\alpha$, 
the proposition follows. \end{proof}

Our second LLT equivalence is a linear relation for LLT polynomials which generalizes the local linear relation of unicellular LLT polynomials given in \cite{Lee21, HNY20}.

\begin{proposition}\label{Prop: linear relation between LLT's}
For a ribbon $R$ of content $[r-1]$, we have the following LLT equivalence
\[
[k]_q \left(C_r^{k-1},R,C_r\right)= q[k-1]_q \left(C_r^k, R\right)+\left(R, C_r^k\right).
\]
\end{proposition}
\begin{proof}
The following linear relation is established
in \cite{Lee21}:
\[
[2]_q\left(C_r,C_{r-1},C_r\right)= q \left(C_r^2, C_{r-1}\right)+\left(C_{r-1}, C_r^2\right).
\]
The only difference between this LLT equivalence and ours is that we replaced $C_{r-1}$ with a ribbon $R$ of content $[r-1]$. Since a cell of content $r$ cannot form an inversion pair with a cell of content less than $r-1$, it suffices to care about the cell of content $r-1$ in $R$ (the last cell). The proof of \cite[Theorem 3.5]{Lee21} applies to show
\[
[2]_q\left(C_r,R,C_r\right)= q \left(C_r^2, R\right)+\left(R, C_r^2\right).
\]
Applying this inductively, as in \cite[Proof of Theorem 3.4]{HNY20}, proves the proposition.
\end{proof}

Our third LLT equivalence is a commuting relation between a ribbon and a cell.

\begin{proposition}\label{Prop: commuting relation between a ribbon and a cell}\cite{Mil19}
Let $R$ be a ribbon of content $[r]$ with a descent set $D$. Then we have the following LLT equivalence
\begin{equation}\label{Eq: commuting relation for a ribbon and a cell}
\left(C_r, R\right)=\begin{cases}
q^{-1}\left(R,C_r\right) & \text{  if $r\in D$, and}\\
\left(R,C_r\right) & \text{  otherwise.}
\end{cases}
\end{equation}
\end{proposition}
\begin{proof}

To prove that two linear combinations of tuples of skew partitions are LLT equivalent, it suffices to show that there is a bijection preserving weight, inversion, and content between semistandard tableaux corresponding to those. Since a similar argument also proves the second case, we only prove the first case. 

For semistandard tableaux of shape $(C_r, R)$ and $(R,C_r)$ are the followings tableaux with either of the following four conditions holds: $a>b>c$, $a=b>c$, $b>a>c$, or $b>c\ge a$.
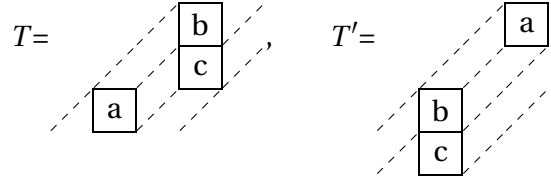
\begin{center}
\ytableausetup{nosmalltableaux}
\ytableausetup{nobaseline}
$T$=\begin{ytableau}
    \none & \none & \none[\diag] & $b$ & \none[\diag]\\
    \none &\none[\diag] &\none[\diag] & $c$ & \none[\diag]\\
    \none[\diag] & $a$ &\none[\diag] & \none[\diag]
\end{ytableau},\qquad
$T'$=\begin{ytableau}
    \none & \none & \none[\diag] & $a$ & \none\\
    \none & \none[\diag] & \none[\diag] & \none[\diag] \\
    \none[\diag] & $b$ &\none[\diag] & \none[\diag]\\
    \none[\diag] & $c$ &\none[\diag]
\end{ytableau}
\end{center}
We let the bijection preserve the fillings of the cells of $R$ of content less than $r-1$, so we omit those cells. In Table~\eqref{Table 1}, we give the $q$-weight of each side of \eqref{Eq: commuting relation for a ribbon and a cell}, $\inv(T)$ and $\inv(T')-1$ for each case.
\begin{equation}\label{Table 1}
    \begin{tabular}{|c|c|c|c|c|c|}
    \hline
         &  $a>b>c$ & $a=b>c$ & $b>a>c$ & $b>c\ge a$ \\
    \hline
        $T$ & 1  & 0 & 0 & 0 \\
    \hline
    $q^{-1}T'$ & 0 & 0 & 1 & 0 \\ \hline
    \end{tabular}
\end{equation}

Let us define a bijection sending $T$ to $T'$ if $a,b$ and $c$ satisfies the second or the fourth case ($a=b>c$ or $b>c\ge a$), and sending $T$ to a tableau obtained from $T'$ by switching $a$ and $b$ otherwise ($a>b>c$ or $b>a>c$). By definition of the bijection, it is obviously weight and content preserving, and \eqref{Table 1} shows that it is also inversion preserving.
\end{proof}

For the last, we need one more commuting relation between dominoes given in \cite{Lee21}.

\begin{proposition}\label{Prop: commuting dominoes}\cite[Lemma 5.5]{Lee21}
Let $V$ and $H$ be vertical and horizontal dominoes of content $[r,r+1]$, respectively. Then we have the following LLT equivalence 
\[
 (V,H) = (H,V).
\]
\end{proposition}

\subsection{Proof of Haglund's conjecture}\label{Sec: a proof of Haglund's conjecture}
Let $\mu=(n-k,k)$. Note that for a subset $D\subseteq \{ (2,j) \in \mu \}$, $\textrm{arm}(D)=\sum_{(2,j)\in D} (k-j)$. Therefore, the multi-$t$-Macdonald polynomial is given by
\[
\widetilde{H}_\mu[X;q,q^{h_k-k}, q^{h_{k-1}-k}, \dots, q^{h_1-k}]=\sum_D q^{\sum_{(2,j)\in D} h_{k+1-j}-2k+j}\llt_{\bmR_\mu(D)}[X;q],
\]
where the sum is over all subsets $D\subseteq \{ (2,j) \in \mu \}$. For a subset $D\subseteq \{ (2,j) \in \mu \}$, we define 
\[
    D^{rev}=\{(2,k+1-j):(2,j)\in D\}.
\]
For example, let $\mu=(7,5)$ and
\[
D=\{(2,1),(2,2),(2,4)\}=
        \begin{ytableau}
        *(black) & *(black) &  & *(black) &\\
        &  &  &  & & & 
        \end{ytableau}.
\]
Then we have
\[
    D^{rev}=\{(2,2),(2,4),(2,5)\}=
        \begin{ytableau}
        *(white) & *(black) &  & *(black) & *(black) \\
        &  &  &  & & & 
        \end{ytableau}.
\]
The tuple of ribbons corresponding to $D$ and $D^{rev}$ are
\[
    \bmR_\mu(D)=(V,V,H,V,H,C,C) \quad \text{ and } \quad \bmR_\mu(D^{rev}) = (H,V,H,V,V,C,C),
\]
where $V$ and $H$ are vertical and horizontal dominoes (of content 1,2) and $C$ is a cell of content 1. Since Proposition~\ref{Prop: commuting dominoes} says that we can swap the vertical and horizontal dominoes ($V$'s and $H$'s), 
\[
\llt_{\bmR_\mu(D)}[X;q]=\llt_{\bmR_\mu(D^{rev})}[X;q].
\]
Thus we have
\begin{align}\label{def:tworowH}
\widetilde{H}_\mu[X;q,q^{h_k-k}, q^{h_{k-1}-k}, \dots, q^{h_1-k}]=\sum_{D}q^{\sum_{(2,j)\in D} h_{j}-k-j+1}\llt_{\bmR_\mu(D)}[X;q].
\end{align}
Our goal is to show that the right-hand side of \eqref{def:tworowH} is the same as $\llt_{D(h_1,\dots,h_k,n,\dots,n)}[X;q]$.

There are two steps for proving this identity using induction. The first step is to show the identity for ``near-staircase" shapes, namely, the case when $h_j=k+j-1+e_j$ for $j=1,\dots,k$, where $e_j$ is either 0 or 1. These partitions are conjugate to partitions containing $(n-k-1,n-k-2,\dots,n-2k)$ and contained in $(n-k,n-k-1,\dots,n-2k+1)$. Proof of the first step is straightforward. When $h_j=k+j-1+e_j$, we have $h_j-k-j+1=e_j$ so by Proposition~\ref{Prop: horizontal + q^a vertical}, (\ref{def:tworowH}) becomes
\[
\sum_D q^{\sum_{c_j \in D}e_j} \llt_{R_\mu(D)}[X;q]= \llt_{(P_{e_1},P_{e_2},\dots,P_{e_k},C_1,C_1,\dots,C_1)}[X;q],
\]
where $P_{1} = (C_1,C_2)$ and $P_0=(C_2,C_1)$ and there are $n-2k$ number of $C_1$'s in total.

The second step is to show that (\ref{def:tworowH}) satisfies both column and row linear relations. Proving the column linear relation is trivial by the formula since a function $q^h_j$ is linear in $h_j$. Therefore, one only needs to prove the row linear relation. Let $f(h_1,\cdots,h_k)$ be the right-hand side of equation (\ref{def:tworowH}). Then we need to show that for any positive integer $i\in [k]$ and $a\leq n$ satisfying $h_{i-1}\leq a$ and $h_{i+2}\geq a+1$, the following holds:
\begin{align}\label{eq:rowlinear1}
 q\cdot f(h_1,\dots,h_{i-1},a,a,h_{i+2},\dots,h_k) +f(h_1,\dots,h_{i-1},a+1,a+1,h_{i+2},\dots,h_k)\\ =
[2]_q f(h_1,\dots,h_{i-1},a,a+1,h_{i+2},\dots,h_k) \label{eq:rowlinear2}
\end{align}
When $(2,i)$ and $(2,i+1)$ are both in $D$, or neither of them is in $D$, the corresponding sum $\sum_{D}q^{\sum_{(2,j)\in D} h_{j}-k-j+1}\llt_{\bmR_\mu(D)}[X;q]$ satisfies (\ref{eq:rowlinear1})=(\ref{eq:rowlinear2}).

When exactly one of $(2,i),(2,i+1)$ is in $D$, the corresponding function
for (\ref{eq:rowlinear1}) is
\begin{align*}
\sum_{D}q^{1+\chi(i,D)(a-k-i+1)+\chi(i+1,D)(a-k-i)}q^{\sum_{(2,j)\in D, j\not= i,i+1} h_{j}-k-j+1}\llt_{\bmR_\mu(D)}[X;q]\\
+\sum_{D}q^{\chi(i,D)(a-k-i+2)+\chi(i+1,D)(a-k-i+1)}q^{\sum_{(2,j)\in D, j\not= i,i+1} h_{j}-k-j+1}\llt_{\bmR_\mu(D)}[X;q]\\
=2[2]_qq^{a-k-i+1}\sum_{D'}q^{\sum_{(2,j)\in D, j\not= i,i+1} h_{j}-k-j+1}\llt_{\bmR_\mu(D')}[X;q]
\end{align*}
where $\chi(i,D)$ is $1$ if $(2,i)\in D$, 0 otherwise, and the sum in the last line runs over all $D'\subset \{(2,j)\in \mu\}$ satisfying $(2,i)\in D'$ and $(2,i+1) \notin D'$. Here we applied Proposition \ref{Prop: commuting dominoes}.

By a similar argument, one can prove
\[2q^{a-k-i+1}\sum_{D'}q^{\sum_{(2,j)\in D, j\not= i,i+1} h_{j}-k-j+1}\llt_{\bmR_\mu(D')}[X;q]\]
is the same as (\ref{eq:rowlinear2}), proving the claim.

We show that proving the two steps is enough to show Haglund's conjecture using induction, where the base case is the conjecture for near-staircase shapes.

By the induction hypothesis, Haglund's conjecture is true for a Dyck path $D(h'_1,h'_2,\ldots,h'_k)$ where $h'_1=k$ or $k+1$, and $k+1\leq h'_i \leq n$ for any $1<i$. Assume that we have any Dyck path $D(h_1,\ldots,h_k)$ satisfying $k\leq h_i \leq n$. There are two cases: 
\begin{enumerate}
    \item $h_2\geq k+2$. In this case, since we know that Haglund's conjecture holds for $h_1=k$ and $k+1$, it holds for any $k\leq h_1 \leq h_2$ by using the fact that both sides of Haglund's conjecture satisfy the column linear relation.
    \item $h_2=k$ or $k+1$. In this case, consider the largest index $i$ satisfying $h_i=k$. If there is no $i$ or $i=1$, the corresponding two Dyck paths satisfy Haglund's conjecture by an induction. If $i>1$, then we use two previous paths and the row linear relation to show that Haglund's conjecture holds for any $i$.
\end{enumerate}

\section{Stretching symmetries}\label{Sec: Stretching}

The starting point for this project was observing a certain `stretching symmetry'
satisfied by the modified Macdonald polynomials $\widetilde{H}_{\mu}[X;q,t]$.
After the authors posted an earlier version of this paper was posted to the arxiv, Mark Haiman
\cite{Hai22} informed the authors that this observation follows from the work
of Garsia and Tesler \cite{GT96}. We reproduce this derivation here.

Given a partition $\mu = (\mu_1, \mu_2, \dots )$, let $k \mu$ be the partition
$k \mu = (k \mu_1, k \mu_2, \dots )$ obtained by multiplying each part of $\mu$
by $k$.
This stretching operation on partitions lifts to the following symmetry of 
the modified Macdonald polynomials at $t=q^k$.

\begin{thm}\label{Theorem:  for general shape}\cite{GT96}
Let $\mu$ be a partition and $k$ be a positive integer. Then we have
\[\widetilde{H}_{k\mu}[X;q,q^k]=\widetilde{H}_{k\mu'}[X;q,q^k].\]
\end{thm}

\begin{proof}
    For any partition $\lambda$, let $B_{\lambda}(q,t)$ be the polynomial
    \[
    B_{\lambda}(q,t) := \sum_{c \in \lambda} q^{\coarm_{\lambda}(c)} t^{\coleg_{\lambda}(c)}.
    \]
    For any two partitions $\lambda, \nu$, Theorem 1.1 of \cite{GT96} yields the 
    implication
    \begin{equation}
        \label{garsia-tesler-implication}
        B_{\lambda}(q^r,q^s) = B_{\nu}(q^r,q^s) \quad \Rightarrow \quad
        \widetilde{H}_{\lambda}[X;q^r,q^s] = \widetilde{H}_{\nu}[X,q^r,q^s]
    \end{equation}
    for any positive integers $r, s$. Since $B_{k \mu}(q,q^k) = B_{k \mu'}(q,q^k)$,
    the result follows.
\end{proof}

Like Theorem~\ref{thm: Haglund's conjecture}, 
Theorem~\ref{Theorem:  for general shape} concerns specializing modified Macdonald polynomials
by sending $t$ to a power of $q$.
In a forthcoming paper of the first and second named authors with Donghyun Kim,
a different proof \cite{KLO22+}
of Theorem~\ref{Theorem:  for general shape} will be given.
This proof will show that the equality of $B$-polynomials
$B_{\lambda}(q,q^k) = B_{\nu}(q,q^k)$ is equivalent to the equality 
$\widetilde{H}_{\lambda}[X;q,q^k] = \widetilde{H}_{\mu}[X;q,q^k]$ of
specialized Macdonald 
polynomials (so that the converse of the implication \eqref{garsia-tesler-implication} 
holds).

For the rest of this section, we study Theorem~\ref{Theorem:  for general shape}, where the partition $\mu$ is a single column. To be more precise, we give two new proofs of the following result.

\begin{thm}\label{thm: tugging symmetry for 1 column}
For nonnegative integers $k$ and $\ell$, we have
\[\widetilde{H}_{\left(k^\ell\right)}[X;q,q^k]=\widetilde{H}_{(k\ell)}[X;q,q^k]=
\sum_{T\in\syt(n)}q^{\operatorname{maj}(T)} s_{\lambda(T)}\]
where $\lambda(T)$ is the shape of the standard tableau $T$.
\end{thm}

The final equality in Theorem~\ref{thm: tugging symmetry for 1 column} holds because
$H_{(n)}[X;q,t] = \sum_{T\in\syt(n)} q^{\operatorname{maj}(T)} s_{\lambda(T)}$ 
is the (singly) graded Frobenius image of the coinvariant ring attached to the symmetric group $\mathfrak{S}_n$;
see Subsection~\ref{representation-theory-proof} for more details. 
We prove Theorem~\ref{thm: tugging symmetry for 1 column} in two ways:
a combinatorial argument using LLT-equivalences and an algebraic argument involving Garsia-Haiman
modules.

\subsection{Combinatorial proof of Theorem~\ref{thm: tugging symmetry for 1 column}}
\label{combinatorial-proof}
Let $\mu=\left(k^\ell\right)$. By Theorem~\ref{thm: LLT expansion of Mac}, we have an LLT expansion of the modified Macdonald polynomial of $\mu$ at $t=q^k$ by
\begin{equation}\label{Equation: LLT sum of Macdonald of rectangle}
    \widetilde{H}_{\mu}(x;q,q^k)=\sum_{D\subseteq \{(i,j)\in\mu:1<i\}} q^{k\maj(D)-\operatorname{arm}(D)} \llt_{\bmR_{\mu}(D)}[X;q].
\end{equation}
Note that for the cells in $\{(i,j): 1<i\le \ell, 1\le j \le k\}$, the $q$-statistics
\[
k\maj-\arm
\]
in \eqref{Equation: LLT sum of Macdonald of rectangle} are given by $1,2, \dots, k$ for the cells in the top row, $k+1, k+2, \dots, 2k$ for the cells in the second top row, and so on.

Choose a subset $E\subseteq \{(i,j)\in\mu: 1<i<\ell\}$ consisting of cells not in the first and the last row and take a partial sum of the right-hand-side of \eqref{Equation: LLT sum of Macdonald of rectangle} over subsets $D\subseteq \{(i,j)\in\mu: 1<i\}$, where $D$ restricted to $\{(i,j)\in\mu: 1<i<\ell\}$ equals to $E$. This partial sum gives $q^{k\maj(E)-\arm(E)}$ times the following sum
\begin{equation}\label{Equation: the first partial sum}
    \sum_{D}q^{\sum_{(\ell,j)\in D}j} \llt_{\bmR_\mu(D)}[X;q],
\end{equation}
where the sum is over subsets $D$ whose restriction to $\{(i,j)\in\mu: 1<i<\ell\}$ equals to $E$.

We claim that the summation in \eqref{Equation: the first partial sum} is LLT equivalent with a \emph{single} tuple of ribbons:
\begin{equation}\label{Eq: sum over D|=E}
\sum_{D}q^{\sum_{(\ell,j)\in D}j} \bmR_{\mu}(D)=
{\left(\bmR_{\left(k^{\ell-1}\right)}(E), C_\ell^k\right)},
\end{equation}
where the sum in the left-hand-side is over subsets $D$ whose restriction to $\{(i,j)\in\mu: 1<i<\ell\}$ equals to $E$. 
We prove this claim by induction on $k$. For the initial case, assume $k=1$. By Proposition~\ref{Prop: horizontal + q^a vertical} for $\alpha=1$, we have 
\[
\bmR_{(1^\ell)}(E) + q 
\bmR_{(1^\ell)}\left(E\cup\{\ell\}\right)=
\left(R_{[\ell-1]}\left(E\right), C_\ell\right),
\]
which proves the claim for $k=1$. 

Assume $k>1$. Then we have
\begin{align*}
    &\sum_{D}q^{\sum_{j:(\ell,j)\in D}j} \bmR_{\mu}(D)\\
    &=\left(\bmR_{\left((k-1)^{\ell-1}\right)}\left(\{(i,j)\in E: 1\le i \le k-1\}\right), C_\ell^{k-1},
    R_{[\ell]}(E^{(k)})\right)\\
    &\qquad\qquad\qquad+q^k\left(\bmR_{\left((k-1)^{\ell-1}\right)}\left(\{(i,j)\in E: 1\le i \le k-1\}\right), C_\ell^{k-1}, R_{[\ell]}(E^{(k)}\cup \{\ell\})\right)\\
    &=[k]_q\left(\bmR_{\left((k-1)^{\ell-1}\right)}\left(\{(i,j)\in E: 1\le i \le k-1\}\right), C_\ell^{k-1}, R_{[\ell-1]}(E^{(k)}), C_\ell\right)\\
    &\qquad\qquad\qquad-q[k-1]_q\left(\bmR_{\left((k-1)^{\ell-1}\right)}\left(\{(i,j)\in E: 1\le i \le k-1\}\right), C_\ell^{k}, R_{[\ell-1]}(E^{(k)})\right)\\
    &=\left(\bmR_{\left(k^{\ell-1}\right)}\left(E\right), C_\ell^k\right).
\end{align*}
The first equation follows by the induction hypothesis. The second equation follows from Proposition~\ref{Prop: horizontal + q^a vertical} and the third equation follows from Proposition~\ref{Prop: linear relation between LLT's}. This proves the claim.

Sliding the cells $C_\ell$'s to the leftmost part of the diagonal of content $\ell-1$ gives \begin{equation}\label{Eq: sliding diagonals}
\left(\bmR_{\left(k^{\ell-1}\right)}\left(E\right), C_\ell^k\right)=\left(C_{\ell-1}^k, \bmR_{\left(k^{\ell-1}\right)}\left(E\right)\right).
\end{equation}
Recall that by Proposition~\ref{Prop: commuting relation between a ribbon and a cell}, we can swap a cell of content $r$ and a ribbon $R$ of content $[r]$ where a weight $q^{-1}$ is attached in the case the last cell of $R$ is a descent. Thus,
\begin{equation}\label{Eq: commuting ribbons and cells}
    \left(C_{\ell-1}^k,\bmR_{\left(k^{\ell-1}\right)}(E)\right)=q^{-k\left|\{j:(\ell-1,j)\in E\}\right|}\left(\bmR_{\left(k^{\ell-1}\right)}(E),C_{\ell-1}^k\right),
\end{equation}
Combining \eqref{Eq: sum over D|=E}, \eqref{Eq: sliding diagonals} and \eqref{Eq: commuting ribbons and cells}, we conclude that
\begin{equation}\label{Eq: H_mu=sum over E}
\widetilde{H}_\mu[X;q,q^k]=\sum_{E}q^{k \maj(E)-\arm(E)-k|\{j:(\ell-1,j)\in E\}|}\llt_{\left(\bmR_{(k^{\ell-1})}(E),C^k_{\ell-1}\right)}[X;q],
\end{equation}
where the sum is over all subsets $E\subseteq \{(i,j)\in \mu:1<i<\ell\}$. Note that each cell in the top row in $E$ contributes to the $q$-weight 
\[
    {k\maj(E)-\arm(E)-k\left|\{j:(\ell-1,j)\in E\}\right|}
\]
in \eqref{Eq: H_mu=sum over E} by $1,2, \dots,k$ as in the initial case. Therefore we can apply the whole procedure again to obtain
\begin{equation*}
\widetilde{H}_\mu[X;q,q^k]=\sum_{E}q^{k \maj(E)-\arm(E)-2k|\{j:(\ell-2,j)\in E\}|}\llt_{\left(\bmR_{(k^{\ell-2})}(E),C^{2k}_{\ell-1}\right)}[X;q],
\end{equation*}
where the sum is over all $E\subseteq \{(i,j)\in \mu:1<i<\ell-1\}$. By applying this repeatedly, we prove Theorem~\ref{thm: tugging symmetry for 1 column}.

\subsection{Representation-theoretic proof of Theorem~\ref{thm: tugging symmetry for 1 column}}
\label{representation-theory-proof}
For any $\mu \vdash n$, Haiman established \cite{Hai01} the Schur positivity of $\widetilde{H}_{\mu}[X;q,t]$ 
by proving
\begin{equation}
\mathrm{grFrob}(V_{\mu};q,t) = \widetilde{H}_{\mu}[X;q,t]
\end{equation}
where $V_{\mu}$ is the {\em Garsia-Haiman module} \cite{GH93} attached to $\mu$.
The module $V_{\mu}$ is the following subspace
of $\mathbb{C}[X_n, Y_n] := \mathbb{C}[x_1, \dots, x_n, y_1, \dots, y_n]$.  
Fix a bijective labeling $T$ of the boxes of $\mu$ with $1, 2, \dots, n$
and define a polynomial $\delta_{\mu} \in \mathbb{C}[X_n, Y_n]$ by
\begin{equation}
\delta_{\mu} := \varepsilon_n\cdot
 \prod_{c \in \mu} x^{\mathrm{coarm}(c)}_{T(c)} y^{\mathrm{coleg}(c)}_{T(c)}
\end{equation}
where $\mathfrak{S}_n$ acts on $\mathbb{C}[X_n, Y_n]$ diagonally and 
$\varepsilon_n := \sum_{w \in \mathfrak{S}_n} \mathrm{sign}(w) w$ is the antisymmetrizing idempotent.
The Garsia-Haiman module $V_{\mu}$ is the smallest linear subspace of $\mathbb{C}[X_n, Y_n]$ containing $\delta_{\mu}$
which is closed under the partial differentiation operators 
$\partial/\partial x_i$ and $\partial / \partial y_i$ for $1 \leq i \leq n$.
In particular, when $\mu \vdash n$ is a single row or column, we have an isomorphism of singly-graded
$\mathfrak{S}_n$-modules
$V_{\mu} \cong R_n$, where $R_n := \mathbb{C}[X_n]/\langle \mathbb{C}[X_n]^{\mathfrak{S}_n}_+ \rangle$ is the type A coinvariant 
algebra in the $x$-variables.
  
Let $\eta: \mathbb{C}[X_n, Y_n] \twoheadrightarrow \mathbb{C}[X_n]$ be the evaluation
map that fixes $x_i$ and specializes $y_i \mapsto (x_i)^k$. 
For any $\mu \vdash n$, we have an $\mathfrak{S}_n$-module homomorphism
$\varphi_{\mu}: V_{\mu} \rightarrow R_n$ given by the composition
\begin{equation}
\varphi_{\mu}: V_{\mu} \hookrightarrow \mathbb{C}[X_n, Y_n] 
\xrightarrow{ \, \, \eta \, \, } \mathbb{C}[X_n] \twoheadrightarrow R_n
\end{equation}
of including $V_{\mu}$ into $\mathbb{C}[X_n, Y_n]$, evaluating along $\eta$, and then projecting 
onto $R_n$.

\begin{proposition}
\label{phi-is-isomorphism}
If $n = k \ell$ and $\mu = (k^\ell)$ is a rectangle, then $\varphi_{\mu}$ is an isomorphism.
\end{proposition}

Since the (singly) graded Frobenius image of $R_n$ is
\begin{equation}
\mathrm{grFrob}(R_n; q) = \sum_{\lambda \vdash n} \left( \sum_{T \in \mathrm{SYT}(\lambda)} q^{\maj(T)} \right)
\cdot s_{\lambda}[X]
\end{equation}
and $\eta$ evaluates the $y$-variables to degree $k$,
Proposition~\ref{phi-is-isomorphism} implies 
Theorem~\ref{thm: tugging symmetry for 1 column}.  
We prove Proposition~\ref{phi-is-isomorphism} as follows.

\begin{proof}
The domain and codomain of $\varphi_{\mu}$ are both vector spaces of dimension $n!$, so
it is enough to show that the image of $\varphi_{\mu}$ spans $R_n$.  We choose our filling $T$
of the $k$-by-$\ell$ rectangle $\mu$ so that 
\begin{equation}
\delta_{\mu} = \varepsilon_n \cdot \left( \prod_{\substack{0 \leq i \leq k-1 \\ 0 \leq j \leq \ell-1}} 
x_{\ell j+i+1}^i y_{\ell j+i+1}^j \right).
\end{equation}
This corresponds to the `English reading order' standard filling of $\mu$.
The evaluation $\eta(\delta_{\mu})$ of the Vandermonde determinant is the image in $R_n$ of the Vandermonde, i.e., 
\begin{equation} \eta(\delta_{\mu}) = 
\varepsilon_n  \cdot ( x_1^0 x_2^1 \cdots x_n^{n-1} ).
\end{equation}
If we endow monomials in $x_1, \dots, x_n$ with the lex term order {\bf with underlying 
variable order $x_1 < \cdots < x_n$} the leading monomial of $\eta(\delta_{\mu})$ is
$ x_1^0 x_2^1 \cdots x_n^{n-1} $.  We show that any exponent sequence
$(a_1, \dots, a_n)$ with $a_i < i$ is the leading monomial of some polynomial in $\eta(V_{\mu})$.
Since such monomials constitute the standard basis of $R_n$ with respect to the aforementioned term
order, this completes the proof.

Suppose we have a componentwise inequality $(a_1, a_2, \dots, a_n) \leq (0, 1, \dots, n-1)$.
We apply the Euclidean algorithm to any difference $(i-1) - a_i$ to write
$$
(i-1) - a_i = q_i m + r_i,
$$
where $q_i \geq 0$ and $0 \leq r_i < m$.  We have an element of $V_{\mu}$ given by
\begin{equation}
\label{a-element}
(\partial / \partial x_1)^{r_1}  (\partial / \partial y_1)^{q_1}  \cdots 
(\partial / \partial x_n)^{r_n}  (\partial / \partial y_n)^{q_n} (\delta_{\mu}).
\end{equation}
The image of \eqref{a-element} under $\eta$ has leading monomial
$x_1^{a_1} \cdots x_n^{a_n}$, and the argument in the last paragraph completes the proof.
\end{proof} 

If $\mu \vdash n$ is not a rectangle, the restriction of $\eta$ to $V_{\mu}$ is not injective,
so the above argument does not go through.

\section{Concluding remarks}
\label{Sec: Conclusion}

\subsection{Schur positivity}
Theorem~\ref{thm: tugging symmetry for 1 column} is equivalent to the assertion that
\[
\widetilde{H}_{(k\ell)}[X;q,t]-\widetilde{H}_{(k^\ell)}[X;q,t]
\]
is divisible by $(q^k-t)$. The proof in Section~\ref{Sec: A combinatorial proof} does not only show the equality of the modified Macdonald polynomials at $t=q^k$ in Theorem~\ref{thm: tugging symmetry for 1 column} but also proves the LLT positivity, thus Schur positivity of the quotient
\[
\dfrac{\widetilde{H}_{(k\ell)}[X;q,t]-\widetilde{H}_{(k^\ell)}[X;q,t]}{q^k-t}.
\]

For any partition $\mu$,
Theorem~\ref{Theorem:  for general shape} implies that the rational function 
\begin{equation}
\label{eqn:macdonald-quotient}
    \dfrac{\widetilde{H}_{k\mu}[X;q,t]-\widetilde{H}_{k\mu'}[X;q,t]}{q^k-t}
\end{equation}
is a polynomial in $q,t$, and $X$. Surprisingly, some (but not all) of the quotients
\eqref{eqn:macdonald-quotient}
are Schur positive. In particular, SAGE computations suggest that if each cell $c=(i,j)$ of $\mu$ satisfies either
\begin{enumerate}
    \item $c$ is also contained in $\mu'$, or
    \item $c$ is under the main diagonal, i.e. $i<j$,
\end{enumerate}
then \eqref{eqn:macdonald-quotient}
is Schur positive. It is an interesting question to ask for necessary and sufficient conditions for two partitions $\lambda,\mu$, and $k\ge 0$ so that
\eqref{eqn:macdonald-quotient} is Schur positive.

\subsection{Combinatorial formula for Kostka polynomials}

A combinatorial formula for 
$(q,t)$-Kostka polynomials 
is unknown in general, 
and finding one 
is one of the most 
important open problems in algebraic combinatorics. 

We recall that a \emph{standard tableau} of a partition $\lambda\vdash n$ is a semistandard tableau consisting of $1,2, \dots, n$. We denote the set of standard tableaux of shape $\lambda$ by $\syt(\lambda)$. It is well known that $(q,t)$-Kostka polynomial at $(1,1)$ gives
\[
\widetilde{K}_{\lambda,\mu}(1,1)=|\syt(\lambda)|.
\]
Therefore, the most desirable form of a combinatorial formula for $(q,t)$-Kostka polynomials would be given by a generating function for the standard tableaux with two statistics:
\[
    \widetilde{K}_{\lambda,\mu}(q,t)=\sum_{T\in\syt(\lambda)}q^{\operatorname{stat}_q(T)}t^{\operatorname{stat}_t(T)}.
\]

Theorem~\ref{thm: tugging symmetry for 1 column} implies
\[
    \widetilde{K}_{\lambda,\left(k^\ell\right)}(q,q^k)=\sum_{T\in\syt(\lambda)}q^{\maj(T)}.
 \]
This suggests that the modified Macdonald polynomials (or modified $(q,t)$-Kostka polynomials) for rectangle $\mu$ might have more structure. For example, Theorem~\ref{thm: tugging symmetry for 1 column} implies the desirable $(q,t)$-statistics $\operatorname{stat}_q$ and $\operatorname{stat}_t$ such that \[\textrm{stat}_q+k \textrm{stat}_t\] 
and $\maj$ statistics are equidistributed over $\syt(\lambda)$. We hope that this gives a hint to track the $q,t$-statistics for $(q,t)$-Kostka polynomials for rectangles. The $(q,t)$-Kostka polynomials for partitions of 4 are given in the following table.

\begin{table}[h]
     \centering
     \begin{tabular}{|c|c|c|c|c|c|}\hline
      $\mu$ \textbackslash $\lambda$& [4] & [3,1] & [2,2] & [2,1,1] & [1,1,1,1,1] \\\hline
     
      [4] & 1& $q+q^2+q^3$ &  $q^2+q^4$ & $q^3+q^4+q^5$  & $q^6$ \\\hline
      [2,2] & 1&$t+qt+q$ & $t^2+q^2$ & $qt^2+qt+q^2t$ & $q^2t^2$ \\\hline
      [1,1,1,1] & 1 & $t+t^2+t^3$ & $t^2+t^4$ & $t^3+t^4+t^5$ & $t^6$ \\\hline
     \end{tabular}
 \end{table}

\subsection{A common generalization of two main theorems}

There is an intersection between Theorem~\ref{thm: Haglund's conjecture} and Theorem~\ref{thm: tugging symmetry for 1 column}. 
More precisely, 
taking $\mu=(k,k)$ and $h_i=2k$ for all $i$'s in Theorem~\ref{thm: Haglund's conjecture} yields
\begin{equation}\label{Eq:intersection of two theorems}
\widetilde{H}_{(n,n)}[X;q,q^n]=\sum_{T\in\operatorname{SYT}(2n)}q^{\maj(T)}X^T.
\end{equation}
Taking $\mu=(2)$ (or $\mu=(1,1)$) in Theorem~\ref{thm: tugging symmetry for 1 column} also yields \eqref{Eq:intersection of two theorems}. Therefore, it is natural to ask the following question.

\begin{question}
Is there a common generalization of Theorem~\ref{Theorem:  for general shape} and Theorem~\ref{thm: Haglund's conjecture}?
\end{question}

\subsection{A new Mahonian statistic}
For any $k \geq 1$, we define a statistic $\maj_k$
on words $w = w_1 w_2 \dots$ over the positive integers
by
\[
    \maj_{k}(w) := \sum_{0<j-i<k} \chi((i,j)\in \operatorname{Inv}(w)) + \sum_i i \chi((i,i+k)\in \operatorname{Inv}(w))
\] 
where $\operatorname{Inv}(w)$ is the inversion set of 
$w$ and
for a statement $P$, we let $\chi(P) = 1$
if $P$ is true and $\chi(P) = 0$ if $P$ is false.
We recover the classical
major index $\maj_1(w) = \maj(w)$ at $k = 1$.
Theorem~\ref{thm: tugging symmetry for 1 column} and
the formula for modified Macdonald polynomials in 
\cite{HHL05} imply that for any partition
$\mu = (\mu_1, \dots , \mu_k) \vdash n$ and any integer 
$k \geq 1$, we have
\[
    \sum_{w \in W(\mu)} q^{\maj_k(w)} =
    \frac{[n]!_q}{[\mu_1]!_q  \cdots [\mu_k]!_q }
\]
where $W(\mu)$ is the set of words $w = w_1 \dots w_n$
of content $\mu$ and
$[n]!_q := [n]_q [n-1]_q \cdots [1]_q$ is the 
 $q$-analog of $n!$.
In particular, the operators $\maj_k$ for $k \geq 1$
are equidistributed on $W(\mu)$.
Kadell gave \cite{Kad85} a bijective proof of this fact.

\begin{corollary}
\label{cor:flipped-one-row}
For $n=k\ell$, we have 
\[\widetilde{H}_{\left(\ell^k\right)}[X;q^k,q]=\widetilde{H}_{(k\ell)}[X;q,q^k]=\sum_{T\in\syt(n)}q^{\operatorname{maj}(T)}X^T.\]
\end{corollary}

\begin{proof}
Apply Theorem~\ref{thm: tugging symmetry for 1 column}
and either the $q,t$-symmetry
$\widetilde{H}_\mu[X;q,t] =
\widetilde{H}_{\mu'}[X;t,q]$
or the representation-theoretic argument in
Subsection~\ref{representation-theory-proof}.
\end{proof}

We define a new statistic $\maj'_\ell$ for words 
$w=w_1 \dots w_n$ by
\[
    \maj'_\ell(w)=\sum_{i} \left\lceil\dfrac{i}{\ell}\right\rceil \chi((i,i+\ell)\in \operatorname{Inv}(w))
\]
Corollary~\ref{cor:flipped-one-row} and the 
combinatorial formula for $\widetilde{H}_\mu[X;q,t]$
in \cite{HHL05} implies that the composite statistic
\[
k\maj_\ell - (n-1)\maj'_\ell
\]
is equidistributed with the major index on words. 
That is, for any partition $\mu$ of $n$,
\begin{equation}\label{Eq: equidistribution inv_k and maj}
    \sum_{w\in W(\mu)}q^{k\maj_\ell(w) - (n-1)\maj'_\ell(w)}= \frac{[n]!_q}{[\mu_1]!_q  \cdots [\mu_k]!_q }.
\end{equation}

\begin{question}
Is there a bijective proof of \eqref{Eq: equidistribution inv_k and maj}?
\end{question}

\section*{acknowledgement}
The authors would like to thank James Haglund for telling us about his conjecture. The authors also thank Mark Haiman for pointing out that Theorem~\ref{Theorem:  for general shape} follows from the result of Garsia--Tesler. The authors are also grateful to Donghyun Kim for helpful conversations.
S. J. Lee was supported by the National Research Foundation of Korea (NRF) grant funded by the Korean government (MSIT) (No. 2019R1C1C1003473). B. Rhoades was partially supported by NSF Grant DMS-1953781. J. Oh was supported by KIAS Individual Grant (CG083401) at  Korea Institute for Advanced Study.

\bibliographystyle{alpha}
\bibliography{Macdonald}

\end{document}